\documentclass[leqno]{amsart}
\usepackage{amsmath}
\usepackage{mathtools}
\usepackage{amssymb}
\usepackage{amsthm}
\usepackage{graphicx}
\usepackage{enumerate}
\usepackage[mathscr]{eucal}
\usepackage{tikz}
\usetikzlibrary{decorations.text,calc,arrows.meta}
\theoremstyle{plain}

\numberwithin{equation}{section}
\usepackage[english]{babel}
\DeclareMathOperator{\sgn}{sgn}
\DeclareMathOperator{\Int}{Int}
\DeclareMathOperator{\dist}{dist}

\newtheorem{theorem}{Theorem}[section]
\newtheorem{cor}[theorem]{Corollary}
\newtheorem{lemma}[theorem]{Lemma}

\newtheorem{prop}[theorem]{Proposition}
\theoremstyle{definition}
\newtheorem{definition}[theorem]{Definition}
\newtheorem{example}[theorem]{Example}

\usepackage{hyperref}
\hypersetup{
	colorlinks=true,
	linkcolor=blue,
	filecolor=magenta,      
	urlcolor=cyan,
}
\usepackage[pagewise]{lineno}
\usepackage{hyperref}

\begin{document}
	
	\title[Sub-Algebras of $\mathcal{H}ol(\Gamma\cup\Int(\Gamma))$ and zeros of Holomorphic functions]{Geometry of sub-algebras of $\mathbf{\mathcal{H}ol(\Gamma\cup\Int(\Gamma))}$ and zeros of Holomorphic functions}
	
	\author[Bose, Roy, Sain]{Babhrubahan Bose, Saikat Roy, Debmalya Sain}
	\newcommand{\acr}{\newline\indent}

	\address[Bose]{Department of Mathematics\\ Indian Institute of Science\\ Bengaluru 560012\\ Karnataka \\INDIA\\ }
	\email{babhrubahanb@iisc.ac.in}

	\address[Roy]{Department of Mathematics\\ National Institute of Technology Durgapur\\ Durgapur 713209\\ West Bengal\\ INDIA}
	\email{saikatroy.cu@gmail.com}
	
	\address[Sain]{Department of Mathematics\\ Indian Institute of Science\\ Bengaluru 560012\\ Karnataka \\INDIA\\ }
	\email{saindebmalya@gmail.com}
	
	\thanks{The research of Dr. Debmalya Sain is sponsored by Dr. D. S. Kothari Post-doctoral Fellowship under the mentorship of Professor Gadadhar Misra. The research of Mr. Saikat Roy is supported by CSIR MHRD in form of Junior Research Fellowship under the supervision of Professor Satya Bagchi. The research of Mr. Babhrubahan Bose is funded by PMRF Research Fellowship under the supervision of Professor Gadadhar Misra and Professor Apoorva Khare.} 
	
	
	\begin{abstract}
		We study $  \mathcal{H}ol(\Gamma\cup\Int(\Gamma)) $, the normed algebra of all holomorphic functions defined on some simply connected neighborhood of a simple closed curve $\Gamma$ in $\mathbb{C} $, equipped with the supremum norm on $ \Gamma $. We explore the geometry of nowhere vanishing, point separating sub-algebras of $ \mathcal{H}ol(\Gamma \cup \Int(\Gamma)) $. We characterize the extreme points and the exposed points of the unit balls of the said sub-algebras for $\Gamma$ analytic. We also characterize the smoothness of an element in these sub-algebras by using Birkhoff-James orthogonality techniques without any restriction on $\Gamma$. As a culmination of our study, we assimilate the geometry of the aforesaid sub-algebras with some classical concepts of complex analysis and establish a connection between Birkhoff-James orthogonality and zeros of holomorphic functions.
	\end{abstract}
	
	\maketitle

	\section{Introduction}
	
	The aim of the present article is to study the geometry of a family of normed algebras and to establish a connection between Birkhoff-James orthogonality and zeros of holomorphic functions. The mentioned normed algebras consist of holomorphic functions defined on any simply connected neighborhood of a simple closed curve in $\mathbb{C}$. We shall let $\mathcal{H}ol(K)$ denote the normed algebra of all continuous functions on a compact set $K\subset \mathbb{C}$ that can be extended to a holomorphic function on some neighborhood of $K$, equipped with the supremum norm on $K$. In the same spirit, we consider the normed algebra $\mathcal{H}ol(\Gamma\cup\Int(\Gamma))$ for a simple closed curve $\Gamma\subset\mathbb{C}$, where $\Int(\Gamma)$ denotes the simply connected region enclosed by $\Gamma$, i.e.,
	\begin{equation*}
		\mathcal{H}ol(\Gamma\cup\Int(\Gamma)):=\{f:\Gamma\cup\Int(\Gamma)\to\mathbb{C}:\exists~U_f\supset\Gamma\cup\Int(\Gamma)~\text{open},~f~\text{holomorphic~on}~U_f\},
	\end{equation*}
	with the norm defined as:
	\begin{equation*}
		\|f\|:=\sup\limits_{z\in\Gamma}|f(z)|,\;f\in\mathcal{H}ol(\Gamma\cup\Int(\Gamma)).
	\end{equation*}
	
	\noindent Clearly, if $f\in\mathcal{H}ol(\Gamma\cup\Int(\Gamma))$, then $f^{(n)}\in\mathcal{H}ol(\Gamma\cup\Int(\Gamma))$ for every $n\in\mathbb{N}$, where $f^{(n)}$ denotes the $n$-th order derivative of $f$.\par
	A sub-algebra $\mathfrak{A}$ of $\mathcal{H}ol(\Gamma\cup\Int(\Gamma))$ is said to be nowhere vanishing and point separating at $ \Gamma $ if for every $z\in\Gamma$, there exists $f\in\mathfrak{A}$ such that $f(z)\neq 0$, and for every $z_1,z_2\in\Gamma$, there exists $g\in\mathfrak{A}$ such that $g(z_1)\neq g(z_2)$. Throughout this article, we call these sub-algebras as nowhere vanishing point separating sub-algebras of $ \mathcal{H}ol(\Gamma \cup \Int(\Gamma)) $, without any ambiguity.\par
	For a given choice of $\Gamma\subset\mathbb{C}$, define:
	\begin{align}\label{everywhere norming}
		\mathcal{J}(\Gamma):=\{f\in\mathcal{H}ol(\Gamma\cup\Int(\Gamma)):|f(z)|=\|f\|\;\forall z\in\Gamma\},
	\end{align}
	and for a given $f\in\mathcal{H}ol(\Gamma\cup\Int(\Gamma))$, set:
	\begin{align}
		M_f & :=\{z\in\Gamma:|f(z)|=\|f\|\}\label{norming points},\\
		\mathcal{Z}_f & :=\{z\in\Gamma\cup\Int(\Gamma):f(z)=0\}\label{zeros}.
	\end{align}
	For example, if $\Gamma=\{z\in\mathbb{C}:|z|=r\}$ for some $r>0$, then $f\in\mathcal{J}(\Gamma)$ if and only if $f$ is constant or
	\begin{align*}
		f(z)=c\prod\limits_{k=1}^n\left(\frac{z-ra_k}{r-\overline{a_k}z}\right),\;z\in\Gamma\cup\Int(\Gamma),
	\end{align*}
	for some $c\in\mathbb{C}\backslash\{0\}$, $n\in\mathbb{N}$ and $a_k\in\mathbb{D}$ for $1\leq k\leq n$, where $\mathbb{D}$ denotes the open unit disc in $\mathbb{C}$.\par 
	For a given $z\in\mathbb{C}$, the real and the imaginary parts of $z$ are denoted by $\Re(z)$ and $\Im(z)$, respectively. If additionally, $z\neq0$, let $\arg(z)$ denote the real number $\theta\in[0,2\pi)$ such that $z=|z|e^{i\theta}$. Also, the sign function $\sgn:\mathbb{C}\to\mathbb{C}$ is defined by 
	\begin{align*}
		\sgn(z):=
		\begin{cases}
			\frac{z}{|z|},~z\neq0,\\
			0,~z=0.
		\end{cases}
	\end{align*}
	\par
	Given a normed linear space $X$, let $B_X$ and $S_X$ denote the closed unit ball and the unit sphere of $X$ respectively, i.e.,
	\begin{align*}
		B_X=\{x\in X:\|x\|\leq1\},~~~ S_X=\{x\in X:\|x\|=1\}.
	\end{align*}
	Let $X^*$ denote the continuous dual of $X$. For a non-zero element $x\in X$, we denote the collection of all support functionals at $x$ by $J(x)$, i.e.,
	\begin{align*}
		J(x)=\{ \Psi\in S_{X^*}: \Psi(x)=\|x\| \},
	\end{align*}
	Let $Ext(X)$ denote the collection of all extreme points of $B_X$. An element $x\in B_X$ is said to be an exposed point of $B_X$ if $\|x\|=1$ and there exists $\Psi \in J(x)$ such that 
	\begin{align*}
		|\Psi(y)|=\|y\|\;\text{if and only if}\;y=cx\;\text{for some}\;c\in\mathbb{C}.
	\end{align*}
	\noindent We denote the collection of all exposed points of $B_X$ by $Exp(X)$. It is well-known that $Exp(X)\subseteq Ext(X)$. For $x,y\in X$, $x$ is said to be Birkhoff-James orthogonal to $y$, denoted by $x\perp_By$, if
	\begin{align*}
		\|x+\lambda y\|\geq\|x\|~~\textit{for~ all~}\lambda\in\mathbb{C}.
	\end{align*}
	In \cite{james}, James proved that for $x,y\in X$, $x\perp_By$ if and only if $x=0$ or there exists some $\Psi\in J(x)$ such that $\Psi(y)=0$. A non-zero element $x\in X$ is said to be smooth if $J(x)$ is singleton. James \cite{james} proved that a non-zero $x\in X$ is smooth if and only if 
	\begin{align*}
		x\perp_B y,~~  x\perp_B  z\implies x\perp_B (y+z)~\textit{for~all~}y,z\in X.
	\end{align*}
	\par
	In the first section of this article, we characterize the extreme points and the exposed points of any nowhere vanishing, point separating sub-algebra of $\mathcal{H}ol(\Gamma\cup\Int(\Gamma))$ for $\Gamma$ analytic. In the next section, we characterize Birkhoff-James orthogonality in the said sub-algebras and also identify its smooth points but for any simple closed curve $\Gamma$, not necessarily analytic. Characterizing the extreme points, the exposed points and the smooth points of the closed unit ball of a given normed linear space is of fundamental importance in determining the geometry of the space. We refer the readers to \cite{abatzoglou}, \cite{deeb-khalil}, \cite{gra1}, \cite{gra2}, \cite{holub}, \cite{Kadison}, \cite{lima1}, \cite{lima2}, \cite{lind}, \cite{sain1}, \cite{sain2}, \cite{sharir} for some of the illustrative works in this regard.\par
	
	In the final section, we find an interrelation between Birkhoff-James orthogonality in nowhere vanishing, point separating sub-algebras of $\mathcal{H}ol(\Gamma\cup\Int(\Gamma))$ with the zeros of holomorphic functions using some classical concepts of complex analysis.

	\section{Geomtery of $\mathcal{H}ol(\Gamma\cup\Int(\Gamma))$ for $\Gamma$ simple, closed and analytic}
	We begin with a couple of simple propositions. In the first one, we characterize the set $\mathcal{J}(\Gamma)$ for $\Gamma\subset\mathbb{C}$ analytic.
	\begin{prop}\label{J_gamma}
		Let $\Gamma\subset\mathbb{C}$ be a simple closed analytic curve and let $f:\Int(\Gamma)\to\mathbb{D}$ be an onto biholomorphic map. Then\\
		$(i)$ $f$ extends to a biholomorphic function $F:U\to\mathbb{C}$ for some neighborhood $U$ of $\Gamma\cup\Int(\Gamma)$ such that $F(\Gamma\cup\Int(\Gamma))=\overline{\mathbb{D}}$.\\
		$(ii)$ $h\in\mathcal{J}(\Gamma)$ if and only if $h$ is constant or
		\begin{align*}
			h(z)=c\prod\limits_{k=1}^n\left(\frac{F(z)-a_k}{1-\overline{a_k}F(z)}\right),\;z\in\Gamma\cup\Int(\Gamma),
		\end{align*}
		for some $c\in\mathbb{C}\backslash\{0\}$, $n\in\mathbb{N}$ and $a_k\in\mathbb{D}$ for $1\leq k\leq n$. 
	\end{prop}
	\begin{proof}
		The first part of the proposition follows from the observation that for every $z\in\Gamma$, there exists $U_z\subset\mathbb{C}$ open such that $\Gamma\cap U_z$ is a one-sided free arc of $\Gamma$ containing $z$ and hence $f$ can be extended biholomorphically to some neighborhood $V_z$ of $\Gamma\cap U_z$, using the Schwarz reflection principle (see \cite[Theorem 4, p. 235]{ahlfors} for details). The desired $F$ is now obtained by considering the extensions on $V_z$ for every $z\in\Gamma$ and restricting the function on a sufficiently small neighborhood of $\Gamma\cup\Int(\Gamma)$ so that it remains injective. \par
		The second part follows from the observation that $T:\mathcal{H}ol(\overline{\mathbb{D}})\to\mathcal{H}ol(\Gamma\cup\Int(\Gamma))$ given by
		\begin{align}\label{isom}
			T(g)(z):=g\circ F(z),\;z\in F^{-1}(U_g),
		\end{align}
		where $\overline{\mathbb{D}}\subset U_g$ open and $g:U_g\to\mathbb{C}$ holomorphic is an isometric isomorphism and $T(g)\in\mathcal{J}(\Gamma)$ if and only if $|g(z)|=\|g\|$ for every $|z|=1$.
	\end{proof}
	
	The second proposition pertains to the completion of $\mathcal{H}ol(\Gamma\cup\Int(\Gamma))$ for arbitrary simple closed curves $\Gamma\subset\mathbb{C}$ (not necessarily analytic). The result follows easily from the observation that $\mathcal{H}ol(\Gamma\cup\Int(\Gamma))$ contains all the polynomials.
	
	\begin{prop}
		If $\Gamma\subset\mathbb{C}$ is a simple closed curve, then $\mathcal{H}ol(\Gamma\cup\Int(\Gamma))$ is incomplete and its completion is isometrically isomorphic to
		\begin{align*}
			\mathcal{A}(\Int(\Gamma)):=\{f\in\mathcal{H}ol(\Int(\Gamma)):\;f\;\text{extends continuously on}\;\Gamma\cup\Int(\Gamma)\},
		\end{align*}
		equipped with the supremum norm.
	\end{prop}
	\par
	Now, we characterize the extreme points and the exposed points of the closed unit ball of any nowhere vanishing, point separating sub-algebra of $\mathcal{H}ol(\Gamma\cup\Int(\Gamma))$ for $\Gamma$ simple, closed and analytic. We begin with a preliminary lemma.
	
	\begin{lemma}
		\label{ratio1} Let $\Gamma\subset\mathbb{C}$ be a simple closed analytic curve. Let $X = \mathcal{H}ol(\Gamma\cup \Int(\Gamma))$ and let $f\in S_X$. If $z_0\in \Gamma$ is an isolated point of $M_f$ (see \eqref{norming points}) and $f(z_0)=e^{i\theta_0}$ for some $ \theta_0\in [0,2\pi) $, then
		\begin{equation}
			\label{lem1}
			\lim\limits_{\substack{z\to z_0\\ z\in\Gamma}}|z-z_0|^{r}\frac{\left|e^{i\theta_0}-f(z)\right|}{1-|f(z)|}=0,\\
		\end{equation}
		for some natural number $r$.
	\end{lemma}
	\begin{proof}
		Replacing $f$ by $e^{-i\theta_0}f$, we may and do assume that $f(z_0)=1$. Also, let $f$ be holomorphic on the domain $\Omega\supset\Gamma\cup\Int(\Gamma)$. Choose a neighborhood $N$ of $z_0$, contained in $\Omega$ such that $|f(z)|<1$ for every $z\in(\Gamma \cap N)\backslash\{z_0\}.$ Set $g(z):=1-f(z)$ on $N$. Then $g$ is holomorphic on $N$ and has a zero at $z=z_0$. Let $m_0$ be the multiplicity of the zero of $g$ at $z=z_0$. Then there exists a neighbourhood $N'$ of $z_0$ and a holomorphic function $h:N'\to\mathbb{C}$ with $h(z_0)\neq0$ such that:
		\begin{equation*}
			g(z)=(z-z_0)^{m_0}h(z),~z\in N'.\\
		\end{equation*}
		Now, observe that\\
		\begin{equation}\label{degree}
			|z-z_0|^{r}\frac{\left|1-f(z)\right|}{1-|f(z)|} = |z-z_0|^{r+m_0}\frac{|h(z)|}{1-|1-g(z)|},\\
		\end{equation}
		and\\
		\begin{align}\label{newform}
			\frac{1}{1-|1-g(z)|}&=\frac{1+|1-g(z)|}{1-|1-g(z)|^2}\nonumber\\
			&=\frac{1+|1-g(z)|}{2\Re(g(z))-|g(z)|^2}
		\end{align}
        
        Now, let $I\subset \mathbb{R}$ be some interval and $\gamma:I\to\mathbb{C}$ be an analytic parametrization of $\Gamma$ in some neighborhood of $z_0$. Since $g\circ \gamma$ is analytic on $I$, $\Re(g\circ\gamma)$ is a real-analytic function. Without loss of generality, we may and will assume that $0\in I$ and $\gamma(0)=z_0$. Then $\Re(g\circ\gamma)$ has a zero at 0 and hence, there exist a natural number $k$, a neighbourhood $I'\subset I$ of $0$ and a real-analytic function $\nu:I'\to\mathbb{R}$ with $\nu(0)\neq0$ such that 
        \begin{align}\label{realanal}
            g\left(\gamma(t)\right)=t^k\nu(t),~t\in I'.
        \end{align}
        Also, by the analyticity of the function $t\mapsto\gamma(t)-z_0$, there exist a natural number $n$, a neighbourhood $I"\subset I$ of $0$ and an analytic map $\psi:I"\to\mathbb{C}$ with $\psi(0)\neq0$ such that
        \begin{align}\label{finalists}
            \gamma(t)-z_0=t^n\psi(t),~t\in I".
        \end{align}
        Now, to show that \eqref{lem1} holds for some $r\in\mathbb{N}$, applying \eqref{degree}, \eqref{newform},\eqref{realanal} and \eqref{finalists} we need to show:
        \begin{align*}
            \lim\limits_{t\to0}|\gamma(t)-z_0|^{r+m_0}\frac{h\circ\gamma(t)\left(1+|1-g\circ\gamma(t)|\right)}{2t^k\nu(t)-|\gamma(t)-z_0|^{2m_0}|h\circ\gamma(t)|^2}&=0\\
            \Leftrightarrow
            \lim\limits_{t\to0}|t|^{n(r+m_0)}|\psi(t)|^{r+m_0}\frac{h\circ\gamma(t)\left(1+|1-g\circ\gamma(t)|\right)}{2t^k\nu(t)-|t^n\psi(t)|^{2m_0}|h\circ\gamma(t)|^2}&=0\\
        \end{align*}
        Clearly, choosing $r\in\mathbb{N}$ sufficiently large yields the result.
	\end{proof}
	
	We now come to the first of our main results.
	\begin{theorem}\label{extreme}
		Let $\Gamma\subset\mathbb{C}$ be a simple closed analytic curve and let $\mathfrak{A}$ be a nowhere vanishing, point separating sub-algebra of $\mathcal{H}ol(\Gamma\cup\Int(\Gamma))$. Then the following are equivalent:\\
		$(i)\;f\in Exp(\mathfrak{A})$.\\
		$(ii)\;f\in Ext(\mathfrak{A})$.\\
		$(iii)\,\|f\|=1$ and $f\in J(\Gamma)$, i.e., $|f(z)|=1$ for every $z\in\Gamma$.
	\end{theorem}
	\begin{proof}
		$(i)\Rightarrow(ii)$\\
		Elementary as discussed before.\\
		$(ii)\Rightarrow(iii)$\\
		Suppose that $\Gamma\cup\Int(\Gamma)\subset \Omega$ is open and $f:\Omega\to\mathbb{C}$ is holomorphic. If $f\in Ext(\mathfrak{A})$, clearly, $f\in S_{\mathfrak{A}}$. 
		\par We first prove that $f\in J(\Gamma)$ if and only if $M_f$ is infinite. The only if part is clearly trivial and for the if part, let $z_n\in M_f$ be an infinite sequence such that $z_n\to z_0$. Consider an interval $I\subset \mathbb{R}$ such that $0\in I$ and $\gamma:I\to\mathbb{C}$ is an analytic parametrization of $\Gamma$ in some neighbourhood of the point $z_0$ with $\gamma(0)=z_0$. Clearly, $\left|f\circ\gamma(t)\right|^2=\left(\Re\left(f\circ\gamma(t)\right)\right)^2+\left(\Im\left(f\circ\gamma(t)\right)\right)^2$ is real analytic on $I$. Since $|f\circ\gamma(t)|^2=1$ for infinitely many $t\in I$, by the identity theorem for real-analytic functions, $|f\circ\gamma(t)|=1$ for every $t\in I$. The result now follows using an elementary connectedness argument. 
		\par Now, assume for the sake of contradiction, $f\in S_{\mathfrak{A}}$ and $M_f=\{z_1,z_2,\dots,z_n\}$ for some $n\in\mathbb{N}$.  Then by Lemma \ref{ratio1}, there exist $r_k\in\mathbb{N}$ such that
		\begin{align}\label{limits}
			\lim\limits_{\substack{z\to z_k\\ z\in\Gamma}}|z-z_k|^{r_k}\frac{\left|e^{i\theta_k}-f(z)\right|}{1-|f(z)|}=0,
		\end{align}
		where $f(z_k)=e^{i\theta_k}$ for every $1\leq k\leq n$. Also, let $m_k$ denote the multiplicity of the zero $z_k$ of the holomorphic map $e^{i\theta_k}-f(z),z\in \Omega$ for every $1\leq k\leq n$.\par
		Now, let $w_1,w_2\in\Gamma$ and $c_1,c_2\in\mathbb{C}$. Since $\mathfrak{A}$ is nowhere vanishing and point separating, there exist $h_1,h_2,h_3\in\mathfrak{A}$ such that
		\begin{align*} 
			h_1(w_1),h_2(w_2)\neq 0,\quad h_3(w_1)\neq h_3(w_2).
		\end{align*}
		Then 
		\begin{align*}
			h:=\frac{c_1(h_1h_3-h_3(w_2)h_1)}{h_1(w_1)(h_3(w_1)-h_3(w_2))}+\frac{c_2(h_2h_3-h_3(w_1)h_2)}{h_2(w_2)(h_3(w_2)-h_3(w_1))},
		\end{align*}
		is well-defined and belongs to $\mathfrak{A}$ with $h(w_1)=c_1$ and $h(w_2)=c_2$. \par
		Now, fix $z_0\in\Gamma\backslash M_f$. Then, there exist $f_1,f_2,\dots f_n\in\mathfrak{A}$, such that $f_k(z_k)=0$ and $f_k(z_0)=1$ for every $1\leq k\leq n$. Suppose $\Gamma\cup\Int(\Gamma)\subset U_k$ is open such that $f_k: U_k\to\mathbb{C}$ is holomorphic and $\mu_k$ is the multiplicity of the zero $z_k$ of the holomorphic map $f_k$ for every $1\leq k\leq n$. Define $P:\Omega\cap\left(\bigcap\limits_{k=1}^n U_k\right)\to\mathbb{C}$ by
		\begin{align*}
			P(z):=\prod\limits_{k=1}^n\left(f_k(z)\right)^{\left\lceil\frac{m_k+r_k}{\mu_k}\right\rceil}, z\in \Omega\cap\left(\bigcap\limits_{k=1}^n U_k\right),
		\end{align*}
		where $\lceil \alpha\rceil=\min\{n\in\mathbb{Z}:\alpha\leq n\}$ for any $\alpha\in\mathbb{R}$. Then, since the multiplicity of the zero $z_k$ of $P$ is at least $m_k+r_k$, for every $1\leq k\leq n$,
		\begin{align*}
			P(z)=(z-z_k)^{m_k+r_k}G_k(z),\;z\in \Omega\cap\left(\bigcap\limits_{k=1}^n U_k\right),
		\end{align*}
		for some $G_k:\Omega\cap\left(\bigcap\limits_{k=1}^n U_k\right)\to\mathbb{C}$ holomorphic. Note that $G_k(z_k)$ may be zero for some $1\leq k\leq n$. Therefore, for any $1\leq k\leq n$, by \eqref{limits},
		\begin{align*}
			\lim\limits_{\substack{z\to z_k\\z\in\Gamma}}\frac{|P(z)|}{1-|f(z)|}=G_k(z_k)\lim\limits_{\substack{z\to z_k\\z\in\Gamma}}\frac{|z-z_k|^{m_k}}{|e^{i\theta_k}-f(z)|}|z-z_k|^{r_k}\frac{|e^{i\theta_k}-f(z)|}{1-|f(z)|}=0.
		\end{align*}
		Fix $\epsilon>0$. Then there exist $N_1,N_2,\dots,N_n\subset\mathbb{C}$ open such that $z_k\in N_k$ and 
		\begin{align*}
			\frac{|P(z)|}{1-|f(z)|}<\epsilon,\;z\in\Gamma\cap N_k\backslash\{z_k\},\;1\leq k\leq n.
		\end{align*}
		Since $\Gamma\backslash\left(\bigcup\limits_{k=1}^nN_k\right)$ is compact and $z\mapsto \frac{|P(z)|}{1-|f(z)|}$ is continuous on $\Gamma\backslash\left(\bigcup\limits_{k=1}^nN_k\right)$,
		\begin{align*}
			\frac{|P(z)|}{1-|f(z)|}\leq M,\;z\in\Gamma\backslash\left(\bigcup\limits_{k=1}^nN_k\right),
		\end{align*}
		for some $M>0$. Set $K=\max\{M,\epsilon\}.$ Then
		\begin{align*}
			|P(z)|\leq K(1-|f(z)|),\;z\in\Gamma.
		\end{align*}
		Hence,
		\begin{align*}
			|f(z)|+\left|\frac{1}{K}P(z)\right|\leq1,\;z\in\Gamma.
		\end{align*}
		Clearly, $\frac{1}{K}P\in\mathfrak{A}$ and $P(z_0)=1$. Hence $P\not\equiv 0$ proving $f\notin Ext(\mathfrak{A})$.\\
		$(iii)\Rightarrow (i)$\\
		Let $f\in S_{\mathfrak{A}}\cap J(\Gamma)$. Fix $\{z_n:n\in\mathbb{N}\}\subset M_f=\Gamma$ such that $z_n\to z_0$ for some $z_0$. Consider $\Psi:\mathfrak{A}\to\mathbb{C}$ given by
		\begin{align*}
			\Psi(h):=\sum\limits_{n=1}^\infty\frac{\overline{\sgn(f(z_n))}}{2^n}h(z_n),\;h\in\mathfrak{A}.
		\end{align*}
		Then clearly, $\Psi$ is unit norm functional on $\mathfrak{A}$ and $\Psi(f)=\|f\|$. Also, $|\Psi(h)|=\|h\|$ for some $h\in\mathfrak{A}$ if and only if 
		\begin{align*}
			h(z_n)=e^{i\theta_0}\|h\|f(z_n),
		\end{align*}
		for every $n\in\mathbb{N}$ and some fixed $\theta_0\in[0,2\pi)$. Since $f$ and $h$ are holomorphic on some connected neighborhood of $\Gamma\cup\Int(\Gamma)$, by identity theorem, $h=e^{i\theta_0}\|h\|f$ proving $f\in Exp(\mathfrak{A})$.
	\end{proof}

	Clearly, if $\Gamma$ is not analytic, there exist functions $f\in Ext(\mathcal{H}ol(\Gamma\cup\Int(\Gamma)))$ such that $M_f\neq\Gamma$. The following simple example illustrates this:
	\begin{example}
		Let $\Gamma:=\{e^{it}:t\in[0,\pi]\}\cup[-1,1]$. Then any monomial $h_n:=z^n,~z\in\mathbb{C}$ is not in $\mathcal{J}(\Gamma)$ but is an extreme point of the closed unit ball of $\mathcal{H}ol(\Gamma\cup\Int(\Gamma))$, since $M_{h_n}=\{e^{it}:t\in[0,\pi]\}$ is an infinite set and therefore if $h_n=\frac{1}{2}(f+g)$ for some $f,g\in S_{\mathcal{H}ol(\Gamma\cup\Int(\Gamma))}$, then $h_n=f=g$ on $M_{h_n}$ proving $f=h_n=g$ by the identity theorem.
	\end{example}
	
	Proposition \ref{J_gamma} and Theorem \ref{extreme} completely characterize the extreme points of the closed unit ball of $\mathcal{H}ol(\Gamma\cup\Int(\Gamma))$ for simple closed analytic curves $\Gamma\subset\mathbb{C}$ as $B\circ F$ where $B$ is a finite Blaschke product or a unimodualr constant function on $\mathbb{D}$ and $F$ is a holomorphic extension of the Riemann map $f:\Int(\Gamma)\to\mathbb{D}$ on some neighborhood of $\Gamma\cup\Int(\Gamma)$. Let us recall the characterization of the extreme points of the closed unit ball of the disc algebra $\mathcal{A}(\mathbb{D})$ from \cite[p.139]{hoffman} given by: 
	\begin{theorem}
		A function $f\in Ext(\mathcal{A}(\mathbb{D}))$ if and only if $|f|\leq1$ on $\overline{\mathbb{D}}$ and 
		\begin{equation*}
			\int\limits_0^{2\pi}\ln{(1-|f(e^{it})|})dt=-\infty.
		\end{equation*}
	\end{theorem}
	Hence clearly, if $f$ is an extreme point of the closed unit ball of $\mathcal{H}ol(\overline{\mathbb{D}})$, $f\in Ext(\mathcal{A}(\mathbb{D}))$. Also observing that the map $T$ defined in \eqref{isom} can be extended to an isometric isomorphism from $\mathcal{A}(\mathbb{D})$ onto $\mathcal{A}(\Int(\Gamma))$, we can conclude the following corollary:
	\begin{cor}
		If $\Gamma\subset\mathbb{C}$ is a simple closed analytic curve, any extreme point of the closed unit ball of any nowhere vanishing, point separating sub-algebra of $\mathcal{H}ol(\Gamma\cup\Int(\Gamma))$ is an extreme point of the closed unit ball of its completion $\mathcal{A}(\Int(\Gamma))$.
	\end{cor}
	The converse of this result is not true, i.e., there exist extreme points of the closed unit ball of the disc algebra $\mathcal{A}(\mathbb{D})$ that are not holomorphic on any neighborhood of $\overline{\mathbb{D}}$ as illustrated by the following example. Note that any function $f\in Ext(\mathcal{A}(\mathbb{D}))$ which is holomorphic in some neighborhood of $\overline{\mathbb{D}}$ must be an extreme point of the closed unit ball of $\mathcal{H}ol(\overline{\mathbb{D}})$, i.e., in particular, $f$ must be a finite Blaschke product or a unimodular constant.
	\begin{example}
		Let $g:[0,2\pi]\to[0,1)$ be a continuous function such that $g(t)=e^{-n}$ on $\left[\frac{1}{2n},\frac{1}{2n-1}\right]$; $n\in\mathbb{N}$ with $g(0)=g(2\pi)=0$, $g$ is increasing on $[0,1]$ and strictly decreasing on $[1,2\pi]$. Note that $\ln{(1-g(t))}$ is continuous and of bounded variation. Consider $f:\mathbb{D}\to\mathbb{C}$ given by
		\begin{equation*}
			f(z):=\exp{\left[\frac{1}{2\pi}\int\limits_0^{2\pi}\left(\frac{e^{it}+z}{e^{it}-z}\right)\ln{(1-g(t))dt}\right]},~~z\in\mathbb{D}.
		\end{equation*}
		Then $f$ is holomorphic on $\mathbb{D}$ and has absolutely summable Taylor coefficients by Hardy's {Theorem} \cite[p.70]{hoffman}. Hence $f$ can be extended continuously to the boundary of $\mathbb{D}$, i.e., $f\in\mathcal{A}(\mathbb{D})$. Also, $|f(e^{it})|=1-g(t)$ for $t\in[0,2\pi)$ and so $|f|\leq1$ on $\overline{\mathbb{D}}$. Further,
		\begin{align*}
			\int\limits_0^{2\pi}\ln{(1-|f(e^{it})|)}dt=\int\limits_0^{2\pi}\ln{(g(t))}dt\leq\sum\limits_{n=1}^\infty\frac{-n}{2n(2n-1)}=-\infty,
		\end{align*}
		{proving} $f\in Ext(\mathcal{A}(\mathbb{D}))$. However, $f(e^{it})=1$ for $t\in[0,2\pi)$ if and only if $t=0$. Therefore, $f\notin Ext(\mathcal{H}ol(\overline{\mathbb{D}}))$, i.e., $f\notin\mathcal{H}ol(\overline{\mathbb{D}})$.
	\end{example}

	\section{Birkhoff-James orthogonality and characterization of smooth points of $\mathcal{H}ol(\Gamma\cup \Int(\Gamma))$}

	Now, we characterize Birkhoff-James orthogonality of two elements in any nowhere vanishing, point separating sub-algebra of $\mathcal{H}ol(\Gamma\cup \Int(\Gamma))$ for a simple closed curve $\Gamma\subset\mathbb{C}$. We begin by providing two sufficient conditions.
	\begin{prop}\label{argument}
		Let $ \Gamma\subset\mathbb{C} $ be a simple closed curve, and $\mathfrak{A}$ be a nowhere vanishing, point separating sub-algebra of $\mathcal{H}ol(\Gamma\cup \Int(\Gamma))$. Let $ f_1, f_2\in \mathfrak{A}$.\\
		(i) If for every $ \theta \in [0, 2\pi) $ there exists $ z\in M_{f_1} $ such that $\arg~ \frac{f_2(z)}{f_1(z)} = \theta $, then $ f_1\perp_B f_2 $. \\
		(ii) If there exists $ z\in M_{f_1} \cap \mathcal{Z}_{f_2} $, then $ f_1\perp_B f_2 $ (see \eqref{zeros} for definition of $\mathcal{Z}_{f}$).
	\end{prop}
	
	\begin{proof}
		$ (i) $ For $ \lambda = 0 $, $ \| f_1 + \lambda f_2 \| = \| f_1 \| $. If instead, $ \lambda = re^{i\phi} $ for some $ r>0 $ and $ \phi \in (0, 2\pi ] $, then by the hypotheses, there exists $ z\in M_{f_1} $ such that $ \arg~ \frac{f_2(z)}{f_1(z)} = 2\pi - \phi $. Let $ \frac{f_2(z)}{f_1(z)} = r_1e^{i(2\pi -\phi)} $ for some $ r_1>0 $. Then clearly,
		\begin{equation*}
			\| f_1 + \lambda f_2 \|  \geq | f_1(z) + \lambda f_2(z) | \geq \| f_1\|.
		\end{equation*}
		Since $ \lambda\in \mathbb{C}\setminus \{0\} $ was chosen arbitrarily, $ f_1 \perp_B f_2 $.\\
		
		$(ii) $ This is immediate since for $\lambda\in\mathbb{C}$ arbitrary,
		\begin{align*}
			\| f_1 + \lambda f_2 \| \geq | f_1(z) + \lambda f_2(z) |= \| f_1\|.
		\end{align*}
	\end{proof}
	The following corollary follows directly from the first part of the above proposition.
	\begin{cor}\label{orthonormal basis}
		Let $ \Gamma:=\{ z \in \mathbb{C} : |z| = r \} $. Let $ h_j(z):=z^j,\;z\in\mathbb{C} $ for $ j\in \mathbb{N} $. Then $ h_n \perp_B h_m $ for $ m\neq n $ in any nowhere vanishing, point separating sub-algebra of $\mathcal{H}ol(\Gamma\cup\Int(\Gamma))$ containing {all the monomials}. 
	\end{cor}
	In order to completely characterize the Birkhoff-James orthogonality of two elements in any nowhere vanishing, point separating sub-algebra of $\mathcal{H}ol(\Gamma\cup\Int(\Gamma))$, we introduce the following definition.
	
	\begin{definition}
		A subset $ A \subseteq\mathbb{C}^2 $ is said to be an \textit{orthogonality covering set} if 
		\begin{align*}
			\mathbb{C}= \bigcup\limits_{(u,v)\in A}\{ \lambda\in \mathbb{C}: | u+\lambda v | \geq |u| \}.
		\end{align*}
	\end{definition}
	We furnish the following two examples illustrating the idea.
	\begin{example}\label{singleton covering set}
		$ A = \{ (z_1, z_2) \}\subset\mathbb{C}^2 $ is an orthogonality covering set if and only if $z_1=0$ or $ z_2 = 0 $.
	\end{example}
	\begin{example}
		$ A =\{(z_1,z_2),(w_1,w_2)\}\subset\mathbb{C}^2$ is an orthogonality covering set if and only if $\overline{z_1}z_2w_1\overline{w_2}\in(-\infty,0]$.
	\end{example}
	\begin{theorem}\label{characterization}
		Let $ \Gamma\subset\mathbb{C} $ be a simple closed curve, and $\mathfrak{A}$ be a nowhere vanishing, point separating sub-algebra of $\mathcal{H}ol(\Gamma\cup \Int(\Gamma))$.  Let $ f_1, f_2\in \mathfrak{A} $. Then
		\begin{align*}
			f_1\perp_B f_2 ~\Leftrightarrow~ \{ (f_1(z), f_2(z)) : z\in M_{f_1} \} ~\textit{~is an~orthogonality~covering~ set}. 
		\end{align*}
	\end{theorem}
	
	\begin{proof}
		Let $ A= \{ (f_1(z), f_2(z)) : z\in M_{f_1} \} $. First we prove the necessity. Suppose by contradiction that $ A $ is not an orthogonality covering set. Then there exists $ \lambda\in \mathbb{C} $ such that 
		\begin{align*}
			| f_1(z) + \lambda f_2(z) | < |f_1(z)| ~~ \textit{for~ all~ } z\in {M}_{f_1}.
		\end{align*}
		
		\noindent We claim that $ | f_1(z) + \mu\lambda f_2(z) | < |f_1(z)| $ for all $ z\in M_{f_1} $ and for all $ \mu \in (0,1) $. Indeed for $ z_0\in M_{f_1} $ and $ \mu_0\in (0,1) $,
		\begin{align*}
			f_1(z_0) + \mu_0\lambda f_2(z_0) & =  (1-\mu_0)f_1(z_0) + \mu_0(f_1(z_0)+\lambda f_2(z_0)).\\
		\end{align*}
		The claim now follows:
		\begin{align*}
			|f_1(z_0) + \mu_0\lambda f_2(z_0)| & = | (1-\mu_0)f_1(z_0) + \mu_0(f_1(z_0)+\lambda f_2(z_0))|,\\
			& < (1-\mu_0)|f_1(z_0)| + \mu_0|f_1(z_0)|,\\
			& = |f_1(z_0)|.
		\end{align*}
		Now, let $ g: \Gamma\times [-1,1] \longrightarrow \mathbb{R} $ be defined by
		\begin{align*}
			g(z, \mu) := | f_1(z) + \mu \lambda f_2(z)|,\;z\in\Gamma,\;-1\leq\mu\leq1.
		\end{align*}
		Clearly, $ g $ is continuous. Let $ z_0\in {M}_{f_1} $ and $\mu_{z_0}\in(0,1)$. Then from the claim,
		\begin{align*}
			g(z_0, \mu_{z_0}) < \|f_1\|.
		\end{align*}
		From the continuity of $ g $, there exist $ \epsilon_{z_0} > 0 $ and $ \delta_{z_0} > 0 $ such that
		\begin{align*}
			g(z, \sigma) < \| f_1\|~for~all\;\;\; (z, \sigma) \in B(z_0, \epsilon_{z_0}) \times ( \mu_{z_0}-\delta_{z_0}, \mu_{z_0}+\delta_{z_0}).
		\end{align*}
		In particular, $ g(z, \mu_{z_0}) < \| f_1\|~ \mbox { for all}~ z \in B(z_0, \epsilon_{z_0}) $. Thus, clearly from our claim, 
		\begin{align*}
			g(z, \sigma) < \| f_1\|\;for~all\;\;\;(z, \sigma) \in B(z_0, \epsilon_{z_0})\times (0, \mu_{z_0}).
		\end{align*} 
		Similarly, if $ u_0 \in \Gamma \setminus {M}_{f_1}$, then $ g(u_0, 0) < \| f_1 \| $. Again, from the continuity of $ g $, there exist $ \epsilon_{u_0} > 0 $ and $ \delta_{u_0} > 0 $ such that
		\begin{align*}
			g(z, \alpha) < \| f_1\|~for~all\;\;\;(z, \alpha) \in B(u_0, \epsilon_{u_0}) \times ( -\delta_{u_0}, \delta_{u_0}).
		\end{align*}
		Now,
		\begin{align*}
			\Gamma \subseteq \bigcup\limits_{z\in {M}_{f_1}}  B(z, \epsilon_z)  \cup \bigcup\limits_{u\in \Gamma\setminus {M}_{f_1}} B(u, \epsilon_u).
		\end{align*}
		Since $ \Gamma $ is compact, there exist $z_1,z_2,\dots,z_{k_1}\in M_{f_1}$ and $u_1,u_2,\dots,u_{k_2}\in\Gamma\backslash M_{f_1}$ such that\\
		\begin{align*}
			\Gamma \subseteq\bigcup\limits_{i=1}^{k_1}  B(z_i, \epsilon_{z_i}) \cup \bigcup\limits_{j=1}^{k_2} B(u_j, \epsilon_{u_j}).
		\end{align*}
		Choose $ \sigma_0 >0$ such that $\sigma_0<\mu_{z_i}$ and $\sigma_0<\delta_{u_j}$ for all $1\leq i\leq k_1$ and $1\leq j\leq k_2$.
		Clearly, $ (f_1 + \sigma_0 \lambda f_2) \in \mathcal{H}ol(\Gamma\cup \Int(\Gamma)) $. Let $ w_0 \in {M}_{(f_1 + \sigma_0 \lambda f_2)} $. Now, either $ w_0 \in \bigcup\limits_{i=1}^{k_1}  B(z_i, \epsilon_{z_i}) $, or $ w_0\in \bigcup\limits_{j=1}^{k_2} B(u_j, \epsilon_{u_j}) $. However, in either case, from the choice of $ \sigma_0 $, 
		\begin{align*}
			\| f_1 + \sigma_0 \lambda f_2 \| = | f_1(w_0) + \sigma_0 \lambda f_2( w_0 ) | < \| f_1 \|,
		\end{align*}
		proving $ f_1\not \perp_B f_2 $.\\
		We now prove the sufficiency. Let $ \kappa\in \mathbb{C} $ be arbitrary. Since $ A $ is an orthogonality covering set, there exists $ v\in {M}_{f_1} $ such that $ |f_1(v) + \kappa f_2(v) | \geq | f_1(v) | $. Thus,
		\begin{align*}
			\| f_1 + \kappa f_2 \| = \underset{z\in \Gamma}{\sup }|f_1(z) + \kappa f_2(z) | 
			\geq |f_1(v) + \kappa f_2(v) |
			\geq \| f_1 \|.
		\end{align*}
		Since $ \kappa $ was chosen arbitrarily, $ f_1 \perp_B f_2 $.
	\end{proof}
	
	Next, we characterize the smoothness of an element in any nowhere vanishing, point separating sub-algebra of $\mathcal{H}ol(\Gamma\cup \Int(\Gamma))$.

	\begin{theorem}\label{smooth point}
		Let $ \Gamma\subset\mathbb{C} $ be a simple closed curve, and $\mathfrak{A}$ be a nowhere vanishing, point separating sub-algebra of $\mathcal{H}ol(\Gamma\cup \Int(\Gamma))$. Let $ f  \in \mathfrak{A} $ be non-zero. Then $ f $ is a smooth point in $ \mathcal{H}ol(\Gamma\cup \Int(\Gamma)) $ if and only if $ {M}_f $ is a singleton set. 
	\end{theorem}
	
	\begin{proof}
		We first prove the necessity. Suppose by contradiction that $ z_1\neq  z_2\in M_f $. Define $\Psi,\Phi:\mathfrak{A}\to\mathbb{C}$ by\\
		\begin{align*}
			\Psi(h):=\overline{\sgn(f(z_1))}h(z_1),\quad\Phi(h):=\overline{\sgn(f(z_2))}h(z_2),\;h\in\mathfrak{A}.
		\end{align*}
		Then $\Psi$ and $\Phi$ are two support functionals of $f$. Also, $\Psi\neq\Phi$ since $\mathfrak{A}$ is nowhere vanishing and point separating proving $f$ is not a smooth point of $\mathfrak{A}$.\\
		To prove the sufficiency, let $ {M}_f = \{ z_1 \} $. Let $ f\perp_B g $ and $ f\perp_B h $ for some $ g, h \in \mathcal{H}ol(\Gamma\cup \Int(\Gamma)) $. If either of $ g, h $ is zero, then trivially $ f\perp_B (g+h) $. Let $ g,h$ be non-zero. From Theorem \ref{characterization}, $ \{ (f(z), g(z)): z \in {M}_f \} $ is a singleton orthogonality covering set. Therefore, by Example \ref{singleton covering set}, $ g(z_1) = 0 $. A similar argument shows $ h(z_1) = 0 $. Hence $ z_1\in \mathcal{Z}_{g+h} $ and consequently by Proposition \ref{argument}, $ f\perp_B (g+h) $.
	\end{proof}

	\section{$\mathcal{H}ol(\Gamma\cup \Int(\Gamma))$ and zeros of holomorphic functions}
	We begin this section with a simple observation:
	\begin{prop}\label{J-gam}
		Let $ \Gamma\subset\mathbb{C} $ be a simple closed curve. Let $ f\in \mathcal{J}(\Gamma)$ (see \eqref{everywhere norming}). Then either $ f $ is constant on $ \Gamma\cup\Int(\Gamma) $ or $ f $ has a zero enclosed by $ \Gamma $.
	\end{prop} 
	
	\begin{proof}
		Suppose by contradiction that $ f $ is non-zero on $\Omega':=\Int(\Gamma)$. Hence, the minimum modulus principle yields 
		\begin{align*}
			\underset {z\in \Gamma\cup\Omega'}{\min}|f(z)| =\min\limits_{z\in\Gamma}|f(z)|= \| f\|.
		\end{align*} 
		Thus, $ f(\Omega')\subseteq \{ \|f\|e^{i\theta}: \theta \in [0, 2\pi) \}$. Clearly, $f(\Omega')$ is not open in $ \mathbb{C} $ while $\Omega'$ is and hence, $ f $ must be constant on $ \Omega$. 
	\end{proof}

	The first result of this section shows that Birkhoff-James orthogonality in nowhere vanishing, point separating sub-algebras of $\mathcal{H}ol(\Gamma\cup \Int(\Gamma))$ has a deep connection with the zeros of holomorphic functions.
	
	\begin{theorem}\label{Rouche's theorem}
		Let $ \Gamma\subset\mathbb{C} $ be a simple closed curve, and $\mathfrak{A}$ be a nowhere vanishing, point separating sub-algebra of $\mathcal{H}ol(\Gamma\cup \Int(\Gamma))$. Let $ f\in \mathcal{J}(\Gamma)\cap\mathfrak{A} $ and let $ g\in \mathfrak{A}$. Then $ f \not\perp_B g $ implies $ f $ and $ g $ have the same number of zeros enclosed by $ \Gamma $.
	\end{theorem}
	
	\begin{proof}
		Since $ f\not\perp_B g $, there exists $ \lambda \in \mathbb{C}\setminus \{ 0 \} $ such that $ \| f+\lambda g \| < \|f \| $. Now, since $f\in\mathcal{J}(\Gamma)$,
		\begin{align*}
			|f(z)+\lambda g(z)|\leq\|f+\lambda g\|<\|f\|=|f(z)|,\;\forall\;\; z\in\Gamma.
		\end{align*}
		Since, $ f $ and $ -\lambda g $ are holomorphic within and on $ \Gamma $, by Rouche's Theorem, they have the same number of zeros enclosed by $ \Gamma $, establishing the result.
	\end{proof}
	
	The converse of the above theorem is not true as illustrated in the following example:
	\begin{example}
		Let $ \Gamma:=\{ z \in \mathbb{C} : |z| = 1 \} $. Let $ f_1(z) = z $ and let $ f_2(z) = z(z-1) $. Clearly, $ f_1\in \mathcal{J}(\Gamma) $. Since $ 1\in {M}_{f_1} \cap \mathcal{Z}_{f_2}$, by Proposition \ref{argument}, $ f_1\perp_B f_2 $. However, $ f_1 $ and $ f_2 $ have the same number of zeros enclosed by $ \Gamma $.   
	\end{example}
	
	It follows from the above theorem that if $\mathfrak{A}$ is any nowhere vanishing, point separating sub-algebra of $\mathcal{H}ol(\Gamma\cup\Int(\Gamma))$ with $f\in\mathcal{J}(\Gamma)\cap\mathfrak{A}$, then for any $g\in\mathfrak{A}$ with $f$ and $g$ having different number of zeros enclosed by $\Gamma$, $f\perp_Bg$. We record this observation in form of the following theorem.
	
	\begin{theorem}\label{another form}
		Let $ \Gamma\subset\mathbb{C} $ be a simple closed curve, and $\mathfrak{A}$ be a nowhere vanishing, point separating sub-algebra of $\mathcal{H}ol(\Gamma\cup \Int(\Gamma))$. Let $ f\in \mathcal{J}(\Gamma)\cap\mathfrak{A} $ and let $ g\in \mathfrak{A} $. If $ f $ and $ g $ have different number of zeros enclosed by $ \Gamma $, then $ f\perp_B g $.
	\end{theorem}
	
	It is easy to see that Corollary \ref{orthonormal basis} is also a direct consequence of Theorem \ref{another form}. We derive from Theorem \ref{another form} the following important inequality regarding polynomials.
	
	\begin{cor}\label{inequality}
		Let $ \Gamma:=\{ z \in \mathbb{C} : |z| = r \} $. Let $ h_j(z):=z^j,\;z\in\mathbb{C}$ for $j\in\mathbb{N}$. Then
		\begin{align*}
			\underset{\lambda\in \mathbb{C}}{\inf} \underset{z\in \Gamma}{\max}|h_n(z) + \lambda Q_{m}(z) | \geq r^n \textit{~for~ every ~polynomial~} {Q_m} ~\textit{of~ degree~}{m<n}.
		\end{align*}
	\end{cor}
	
	\begin{proof}
		Clearly, $ h_j\in \mathcal{J}(\Gamma) $ for every $ j\in \mathbb{N} $. Now, $ Q_m $ and $ h_n $ have different number of zeros enclosed by $ \Gamma $. Therefore, by Theorem \ref{another form}, $ h_n \perp_B P _m $ in $\mathcal{H}ol(\Gamma\cup\Int(\Gamma))$ for all $ m < n $. Thus, we have for every $\lambda\in\mathbb{C}$,
		\begin{align*}
			\max\limits_{z\in\Gamma}|h_n(z)+\lambda Q_m(z)|=\|h_n+\lambda Q_m\|\geq\|h_n\|=r^n,
		\end{align*}
		and the result follows.
	\end{proof}
	
	Our next result relates Birkhoff-James orthogonality in $\mathcal{H}ol(\Gamma\cup \Int(\Gamma))$ for a simple closed curve $\Gamma\subset\mathbb{C}$ with the fundamental theorem of algebra.
	
	\begin{theorem}\label{Polynomial}
		Let $ \Gamma:=\{ z \in \mathbb{C} : |z| = r \} $ and $ h_j(z) := z^j,\; z\in \mathbb{C} $ for $ j\in \mathbb{N} $. If $ Q(z): = a_nz^n+a_{n-1}z^{n-1}+ \dots + a_0,\, z\in \mathbb{C}$ with $a_n\neq0$ and $ r> \max\{1,\; \frac{1}{|a_n|}{\sum\limits_{k=0}^{n-1} |a_k|}\} $, then $ h_n\not\perp_B P $ in $\mathcal{H}ol(\Gamma\cup\Int(\Gamma))$.
	\end{theorem}
	
	\begin{proof}
		Since $ r> \max\{1,\;\frac{1}{|a_n|}{\sum\limits_{k=0}^{n-1} |a_k|}\} $, 
		\begin{align*}
			\| {a_nh_n}- Q\| &  = \| -(a_{n-1}h_{n-1} + a_{n-2}h_{n-2} + \dots + a_0) \|\\
			& \leq  \| a_{n-1}h_{n-1} \| + \| a_{n-2}h_{n-2} \| + \dots +\|a_1h_1\|+ | a_0 |\\
			& \leq r^{n-1} \{ |a_{n-1} | +| a_{n-2} | + \dots + |a_0| \}\\
			& < r^{n} |a_n|\\
			& = |a_n|\| h_n \|.
		\end{align*}
		Thus, $ \| {h_n}-\frac{1}{a_n} Q\| < \| h_n \| $, giving $ h_n \not \perp_B Q $.
	\end{proof}

	We now obtain the fundamental theorem of algebra:
	
	\begin{cor}[Fundamental Theorem of Algebra]
		A polynomial of degree $ n $ has $ n $ zeros in $ \mathbb{C} $.
	\end{cor}
	
	\begin{proof}
		Let $ Q(z) = a_nz^n+a_{n-1}z^{n-1}+ \dots + a_0 $, $a_n\neq0$ and $ r > \max\{1,\; \frac{1}{|a_n|}{\sum\limits_{k=0}^{n-1} |a_k|}\}$. Set $ \Gamma:=\{ z \in \mathbb{C} : |z| = r \} $ and $ h_j(z) := z^j, \; z\in \mathbb{C} $ for $ j\in \mathbb{N} $. Then by Theorem \ref{Polynomial}, $ h_n\not\perp_B Q $. Since $ h_n \in \mathcal{J}({\Gamma}) $, by Theorem \ref{Rouche's theorem}, $ h_n $ and $ Q $ have the same number of zeros enclosed by $ \Gamma $. Since $ r~ >~ \max~ \{ 1,~ \frac{1}{|a_n|}{\sum\limits_{k=0}^{n-1} |a_k|} \} $ was chosen arbitrarily, $ P $ has $ n $ number of zeros in $ \mathbb{C} $.
	\end{proof}
	
	\par
	
	Finally, we obtain a result regarding Birkhoff-James orthogonality between the $ n $-th order derivatives of two holomorphic functions in $\mathcal{H}ol(\Gamma\cup \Int(\Gamma))$ for a simple closed curve $\Gamma\subset\mathbb{C}$. The result also has a connection between Birkhoff-James orthogonality and zeros of the $ n $-th order derivatives of holomorphic functions.
	
	\begin{theorem}\label{derivative}
		Let $\Gamma_1,\Gamma_2\subset\mathbb{C}$ be two simple closed curves such that $\Gamma_2\subset\Int(\Gamma_1)$ and fix $0<r<\dist(\Gamma_1,\Gamma_2)$. Let $ f, g \in \mathcal{H}ol(\Gamma_1\cup\Int(\Gamma_1)) $ be such that $ g^{(n)} \neq 0 $ for some fixed natural number $ n $ and
		\begin{equation*}
			\underset{z\in \Gamma_1}{\max} |f(z)+\lambda_0 g(z)| ~< ~\frac{r^n}{n!}~\underset{z\in \Gamma_2}{\max}~~| f^{(n)}(z)| \textit{~for~some ~}\lambda_0\in \mathbb{C}. 
		\end{equation*}
		Then $ f^{(n)} \not\perp_B g^{(n)} $ in any nowhere vanishing, point separating sub-algebra of $ \mathcal{H}ol(\Gamma_2\cup\Int(\Gamma_2))$ containing both the functions.  
	\end{theorem}
	Using Theorem \ref{Rouche's theorem}, we immediately derive the following:
	\begin{cor}
		Let $\Gamma_1,\;\Gamma_2,\; r$ be as in Theorem \ref{derivative}. Let $ f, g \in \mathcal{H}ol(\Gamma_1\cup\Int(\Gamma_1))$ be such that
		$ f^{(n)}\in \mathcal{J}({\Gamma_2}) $, $ g^{(n)} \neq 0 $ for some natural number $ n $ and
		\begin{align*}
			\underset{z\in \Gamma_1}{\max} |f(z)+\lambda_0 g(z)| ~< ~\frac{r^n}{n!}~\underset{z\in \Gamma_2}{\max}~~| f^{(n)}(z)| \textit{~for~some ~}\lambda_0\in \mathbb{C}. 
		\end{align*}
		Then $ f^{( n)} $ and $ g^{( n)} $ have the same number of zeros enclosed by $ \Gamma_2 $.
	\end{cor}
	\begin{proof}[Proof of Theorem \ref{derivative}]
		We need to show that there exists some $ \lambda \in \mathbb{C} $ such that 
		\begin{align*}
			\max\limits_{z\in\Gamma_2}|  f^{(n)}(z)+\lambda g^{(n)} (z)| < \max\limits_{z\in\Gamma_2}| f^{(n)}(z)|.
		\end{align*}
		Since $ \Gamma_2 $ is compact,
		\begin{equation*}
			\underset{z\in \Gamma_2}{\max}|  f^{(n)}(z) + \lambda_0 g^{(n)}(z) | = |  f^{(n)}(z_0) + \lambda_0 g^{(n)}(z_0) | \textit{~for~some~}z_0\in \Gamma_2 .
		\end{equation*}
		Denote $\Int(\Gamma_1)$ by $\Omega'$. Let $ \gamma_r $ be the circle of radius $ r $ about $ z_0 $. From hypotheses, $ \gamma_r\subset \Omega' $. Now, by Cauchy's integral formula for the derivative, we have that
		\begin{align*}
			f^{(n)}(z_0) + \lambda_0 g^{(n)}(z_0) = \frac{n!}{2\pi i}\int\limits_{\gamma_r} \frac{f(z) + \lambda_0 g(z)}{(z-z_0)^{n+1}}dz.
		\end{align*}
		Taking modulus on both sides, we get
		\begin{align*}
			| f^{(n)}(z_0) + \lambda_0 g^{(n)}(z_0)|\leq \frac{n!}{r^n}~\underset{z\in \gamma_r}{\max}|{f(z) + \lambda_0 g(z)}|,
		\end{align*}
		since $\int\limits_{\gamma_r}\left|\frac{dz}{(z-z_0)^{n+1}}\right|=\frac{2\pi}{r^n}$. Invoking the maximum modulus principle and noting $ \gamma_r\subset \Omega'$ yields:
		\begin{align*}
			\underset{z\in \gamma_r}{\max}|{f(z) + \lambda_0 g(z)}| \leq \underset{z\in \Gamma_1}{\max}|{f(z) + \lambda_0 g(z)}|.
		\end{align*}
		Thus,
		\begin{align*}
			|f^{(n)}(z_0) + \lambda_0 g^{(n)}(z_0)| < \frac{n!}{r^n}~\underset{z\in \Gamma_1}{\max}|{f(z) + \lambda_0 g(z)}|.
		\end{align*}
		Now, from the hypotheses, 
		\begin{align*}
			\max\limits_{z\in\Gamma_2}| f^{(n)} (z)+ \lambda_0 g^{(n)} (z)| = |f^{(n)}(z_0) + \lambda_0 g^{(n)}(z_0)| < \max\limits_{z\in\Gamma_2}| f^{(n)}(z)|.
		\end{align*}
		Thus, $ f^{(n)} \not\perp_B g^{(n)} $ in any nowhere vanishing, point separating sub-algebra of $ \mathcal{H}ol(\Gamma_2\cup\Int(\Gamma_2))$ containing both the functions. 
	\end{proof}

\end{document}